\documentclass[11pt,oneside]{article}

\usepackage{amsthm, amsmath, amssymb, amsfonts,epsfig}
\usepackage{mathtools}
\usepackage{esint}
\usepackage{epstopdf}

\usepackage{float,graphicx}
\usepackage[font=sl,labelfont=bf]{caption}
\usepackage{subcaption}

\usepackage{enumerate}

\usepackage[colorlinks=true, pdfstartview=FitV, linkcolor=blue,
            citecolor=blue, urlcolor=blue]{hyperref}
\usepackage[usenames]{color}
\definecolor{Red}{rgb}{0.7,0,0.1}
\definecolor{Green}{rgb}{0,0.7,0}

\usepackage[numbers]{natbib}

\usepackage{url}
\makeatletter\def\url@leostyle{%
 \@ifundefined{selectfont}{\def\UrlFont{\sf}}{\def\UrlFont{\scriptsize\ttfamily}}} \makeatother

\usepackage{accents, comment}

\usepackage{bbm}

\usepackage{mathrsfs}

\newtheorem{theorem}{Theorem}

\newtheorem{lemma}[theorem]{Lemma}
\newtheorem{corollary}[theorem]{Corollary}

\theoremstyle{definition}
\newtheorem{definition}[theorem]{Definition}

\theoremstyle{remark}
\newtheorem{remark}[theorem]{Remark}

\numberwithin{equation}{section}
\numberwithin{theorem}{section}

\def\cD{\mathcal{D}}

\def\cM{\mathcal{M}}

\def\bC{\mathbb{C}}

\def\bN{\mathbb{N}}

\def\bR{\mathbb{R}}

\newcommand{\norm}[1]{ \| #1 \| }       

\title{On persistence of spatial analyticity in the hyper-dissipative Navier-Stokes models \footnote{to appear in \emph{IUMJ}}}

\author{
Aseel Farhat\\
\small  Florida State University \\[-0.6ex]
\small  and  \\[-0.6ex]
\small  University of Virginia \\[-0.6ex]
\small  \url{afarhat@fsu.edu}
\and 
Zoran Gruji\'c\\
\small  The University of Alabama at Birmingham \\[-0.6ex]
\small  \url{zgrujic@uab.edu}
}

\begin{document}
\maketitle

\begin{abstract}
The goal of this note is to demonstrate that as soon as the hyper-diffusion
exponent is greater than one, a class of finite time blow-up scenarios consistent
with the analytic structure of the flow (prior to the possible blow-up time) can be
ruled out. The argument is self-contained, in spirit of the regularity theory of the
hyper-dissipative Navier-Stokes system in ‘turbulent regime’ developed by Gruji\'c
and Xu.
\end{abstract}

\section{Introduction}\label{intro}

3D hyper-dissipative (HD) Navier-Stokes (NS) system in $\mathbb{R}^3 \times (0, T)$ reads
\begin{align}
&u_t+ (u\cdot\nabla) u = -(-\Delta)^\beta u-\nabla p               ,\\
&\textrm{div}\ u=0,\\              
&u(\cdot,0)=u_0(\cdot)                 
\end{align}
where an exponent $\beta > 1$ measures the strength of the hyper-diffusion, the vector field $u$ is the velocity of the fluid and the scalar field
$p$ the pressure. 

\medskip

It has been known since the work of J.L. Lions (cf. \citet{Lions1959, Lions1969}) that the 3D HD NS system does not permit a spontaneous formation of singularities
as long as $\beta \ge \frac{5}{4}$. Note that for $\beta = \frac{5}{4}$ the scaling invariant level meets the
energy level, i.e., the system is in the critical state. In contrast, the question of whether a singularity can form in the super-critical regime,
$1 < \beta < \frac{5}{4}$ remains open.

\medskip

In a recent work \citet{Grujic2020}, the authors showed that as soon as $\beta > 1$, and the flow is in a suitably defined `turbulent regime', no 
singularity can form. In particular, the approximately self-similar blow-up -- a leading candidate for a finite time blow-up -- was ruled out
for all HD models.

\medskip

In this short note we rule out (as soon as $\beta > 1$) a class of analytic blow-up profiles. The analytic structure in view is a perturbation
of the geometric series (the radius of analyticity shrinking to zero as the flow approaches the singular time), allowing for precise estimates 
on the derivatives of all orders. The setting is as follows. 
Suppose that the initial datum $u_0$ is in $L^\infty$, denote by $T^*$ the first singular time, and by $x^*$ an isolated (spatial) singularity
at $T^*$. Then, for any $t$ in $(0, T^*)$ the solution $u(t)$ is spatially analytic (see, e.g., \citet{Guberovic2010}) and -- for each of the velocity
components $u^i$ -- we can write the following expansion

\[
 u^i(x, t) = \sum_{k=0}^\infty  \sum_{|\alpha|=k} c^i_{\alpha, k}(t) (x-x^*)^\alpha
\]
where $\alpha$ is the multi-index, $\alpha = (\alpha_1, \alpha_2, \alpha_3)$.

\medskip

Henceforth, we make two assumptions on the flow near $(x^*, T^*)$, the first one spells out the analytic structure, and the second one 
stipulates that the singularity build up is focused.

\medskip

\noindent (A1)  Suppose that there exist constants $\epsilon > 0$ and $M > 1$ such that for any 
$t$ in $(T^*-\epsilon, T^*)$,
$c^i_{\alpha,k}(t) = \delta^i_{\alpha,k}(t) \, \frac{1}{\rho(t)^k}$ where $\rho > 0, 
\rho \to 0$ as $t \to T^*$ and for $k \neq 0$
$\displaystyle{\frac{1}{M^k} \le \delta^i_{\alpha,k} \le M^k}$ while $c^i_0 \to \infty$ as $t \to T^*$. 
Assume that the building block functions are such that the resulting coefficient functions
$c^i_{\alpha, k}$ are monotone -- more precisely, increasing (without bounds) in $t$ (a `runaway train'
scenario).

\medskip

\noindent (A2)  The blow-up is focused, i.e., in a spatial 
neighborhood of $x^*$, say $\cal{N}$,

\[
 \|D^{(k)}u(t)\|_{L^\infty(\cal{N})} = \ |D^{(k)}u(x^*, t)|
\]
for all $t$ in $(T^*-\epsilon, T^*)$.

\medskip

The following is the main result.

\begin{theorem}
Let $\beta > 1$, $u_0 \in L^\infty \cap L^2$, and suppose that (A1) and (A2) hold. Then $T^*$ is not a singular time, and the solution 
$u$ can be continued analytically past $T^*$.
\end{theorem}

Let us remark that the above theorem can be derived as a consequence of the general theory presented in \citet{Grujic2020}.
The main point of this note is to present a short, self-contained argument tailored to the class
of potential blow-up profiles in view.


\section{Preliminaries}

The purpose of this section is to review some concepts and results from the general theory of controlling the
$L^\infty$-fluctuations via sparseness of the regions of intense fluid activity presented in 
\citet{Grujic2013, Bradshaw2019, Grujic2020}.

\bigskip

The first one is a local-in-time spatial analyticity result focusing on derivatives of order $k$.

\begin{theorem}\label{th:MainThmVelHyp}[\citet{Grujic2020}]
Let $\beta > 1, u_0 \in L^2$, $k$ a positive integer and $D^{i}u_0 \in L^\infty$ for $0 \le i \le k$. Fix a constant $M>1$ and let
\begin{align}
T_*&=\min\left\{\left(C_1(M) 2^{k} \left(\|u_0\|_2\right)^{k/(k+\frac{3}{2})} \left(\|D^{k}u_0\|_\infty)\right)^{\frac{3}{2}/(k+\frac{3}{2})} \right)^{-\frac{2\beta}{2\beta-1}}, \right. \notag
\\
&\qquad \left. \left(C_2(M) \left(\|u_0\|_2\right)^{(k-1)/(k+\frac{3}{2})} \left(\|D^{k}u_0\|_\infty\right)^{(1+\frac{3}{2})/(k+\frac{3}{2})} \right)^{-1} \right\} \label{eq:TimeLengthHyp}
\end{align}
where $C(M)$ is a constant depending only on $M$. Then there exists a solution
\begin{align*}
u\in C([0,T_*),L^2) \cap C([0,T_*),C^\infty)
\end{align*}
of the 3D HD NS system such that for every $t\in (0,T_*)$ $u$ is a restriction of an analytic function $u(x,y,t)+iv(x,y,t)$ in the region
\begin{align}\label{eq:AnalDomHyp}
\cD_t=: \left\{x+iy\in\bC^3\ \big|\ |y|\le c \, t^{\frac{1}{2\beta}}\right\}\ .
\end{align}
Moreover, $D^ju\in C([0,T_*),L^\infty)$ for all $0\le j\le k$ and
\begin{align}
&\underset{t\in(0,T)}{\sup}\ \underset{y\in\cD_t}{\sup} \|u(\cdot,y,t)\|_{L^2} + \underset{t\in(0,T)}{\sup}\ \underset{y\in\cD_t}{\sup} \|v(\cdot,y,t)\|_{L^2}\le M \|u_0\|_2
\\
&\underset{t\in(0,T)}{\sup}\ \underset{y\in\cD_t}{\sup} \|D^{k}u(\cdot,y,t)\|_{L^\infty} + \underset{t\in(0,T)}{\sup}\ \underset{y\in\cD_t}{\sup} \|D^{k}v(\cdot,y,t)\|_{L^\infty}\le M \|D^{k}u_0\|_\infty\ .
\end{align}
\end{theorem}


\bigskip

Next we recall definitions of what is meant by local `sparseness  at scale' in this context (\citet{Grujic2013}).

\begin{definition}
For a spatial point $x_0$ and $\delta\in (0,1)$, an open set $S$ is 1D $\delta$-sparse around $x_0$ at scale $r$ if there exists a unit vector $\nu$ such that
\begin{align*}
\frac{|S\cap (x_0-r\nu, x_0+r\nu|}{2r} \le \delta\ .
\end{align*}
\end{definition}

\bigskip

The volumetric version is as follows.

\begin{definition}
For a spatial point $x_0$ and $\delta\in (0,1)$, an open set $S$ is 3D $\delta$-sparse around $x_0$ at scale $r$ if
\begin{align*}
\frac{|S\cap B_r(x_0)|}{|B_r(x_0)|} \le \delta\ .
\end{align*}
\end{definition}

\bigskip

(It is straightforward to check that $3$-dimensional $\delta$-sparseness at scale $r$ implies 1D $(\delta)^\frac{1}{3}$-sparseness at scale $r$;
 the converse is false.)

\bigskip

The following result, a regularity criterion, is a $k$-level version of the vorticity result presented in \citet{Grujic2013, Bradshaw2019}.

\vspace{-0.05in}
\begin{theorem}\label{th:SparsityRegDk_Hyp}[\citet{Grujic2020}]
Let $\beta > 1$, $u_0\in L^\infty\cap L^2$, and $u$ in $C([0,T^*), L^\infty)$ where $T^*$ is the first possible blow-up time. Let $s$ be an escape time for $D^{k}u$, and suppose that there exists a temporal point
\begin{align*}
&\qquad\quad t=t(s)\in \left[s+\frac{1}{c_1(M, k, \beta, \|u_0\|_2) |D^{k}u(s)\|_\infty^{\frac{3}{2k+3}\frac{2 \beta}{2\beta-1}}},\ s+\frac{1}{2 c_1(M, k, \beta, \|u_0\|_2)\|D^{k}u(s)\|_\infty^{\frac{3}{2k+3}\frac{2\beta}{2\beta-1}}}\right]
\end{align*}
such that for any spatial point $x_0$, there exists a scale $r \le \frac{1}{c_2(M, k, \beta, \|u_0\|_2) )\|D^{k}u(t)\|_\infty^{\frac{3}{2k+3}\frac{1}{2\beta-1}}}$ with the property that the super-level set
\begin{align*}
&\qquad\quad V^{i,\pm}=\left\{x\in\bR^3\ |\ (D^{k}u)_i^\pm(x,t)>\frac{1}{2M} \|D^{k}u(t)\|_\infty\right\}
\end{align*}
is 1D $\delta$-sparse around $x_0$ at scale $r$; here the index $(i,\pm)$ is chosen such that $|D^{k}u(x_0,t)|=(D^{k}u)_i^\pm(x_0,t)$, 
$M$ and $\delta$ satisfy \begin{align*}
\frac{1}{2} h+(1-h) M = 1,\qquad h=\frac{2}{\pi}\arcsin\frac{1-\delta^2}{1+\delta^2}, \qquad \qquad \frac{2M}{2M+1} <\delta<1
\end{align*}
(e.g., one can take $\delta=\frac{3}{4}$ for a suitable $1 < M < \frac{3}{2}$), and $c_1, c_2$ are derived from the constants in 
Theorem \ref{th:MainThmVelHyp}.

Then, 
$$ \|D^k u(t)\|_\infty \leq \|D^k u(s)\|_\infty,$$
contradicting $s$ being an escape time for $D^{k}u$,
and there exists $\gamma>0$ such that $u\in L^\infty((0,T^*+\gamma); L^\infty)$, i.e. $T^*$ is not a blow-up time.
\end{theorem}


The lemma below is the Sobolev $W^{-k,p}$-version of the volumetric sparseness results in \citet{Farhat2017} and \citet{Bradshaw2019},
which -- in turn -- are vectorial versions of the semi-mixedness lemma in \citet{Iyer2014}

\begin{lemma}\label{le:HkSparse}
Let $r\in(0,1]$ and $f$ a bounded function from $\bR^d$ to $\bR^d$ with continuous partial derivatives
of order $k$. Then, for any tuple $(k,\lambda, \delta, p)$, $k\in\bN^d$ with $|k|=k$, $\lambda\in (0,1)$, $\delta\in(\frac{1}{1+\lambda},1)$ and $p>1$, there exists $c^*(k,\lambda,\delta,d,p)>0$ such that if
\begin{align}\label{eq:ZalphaCond}
\|D^{k} f\|_{W^{-k,p}} \le c^*(k,\lambda,\delta,d,p)\ r^{k+\frac{d}{p}} \|D^{k} f\|_\infty
\end{align}
then each of the super-level sets
\begin{align*}
S_{k,\lambda}^{i,\pm}=\left\{x\in\bR^d\ |\ (D^{k} f)_i^\pm(x)>\lambda \|D^{k} f\|_\infty\right\}\ , \qquad 1\le i\le d, \quad k\in\bN^d, |k| =k
\end{align*}
is $r$-semi-mixed with ratio $\delta$.
\end{lemma}

This leads to the following \emph{a priori} sparseness result for any $\beta \ge 1$.

\begin{theorem}\label{th:LerayZalpha}\citet{Grujic2020}
Let $u$ be a Leray solution (a global-in-time weak solution satisfying the global energy inequality), and assume that $u$ is in $C((0,T^*), L^\infty)$ for some $T^*>0$. Then for any $t\in (0,T^*)$ the super-level sets
\begin{align*}
&\qquad\quad S_{k,\lambda}^{i,\pm}=\left\{x\in\bR^3\ |\ (D^{k} u)_i^\pm(x)>\lambda \|D^{k} u\|_\infty\right\}\ , \qquad 1\le i\le 3, 
\end{align*}
are 3D $\delta$-sparse around any spatial point $x_0$ at scale
\begin{align}\label{eq:kNaturalScale}
r^*_k(t)=c(\|u_0\|_2) \frac{1}{\|D^{k} u(t)\|_\infty^{2/(2k+3)}}
\end{align}
provided $r^* \in (0, 1]$ and with the same restrictions on $\lambda$ and $\delta$ as in the preceding lemma.
\end{theorem}

\medskip

To summarize, at this point, the \emph{a priori} scale of sparseness  \emph{vs.} the scale of the analyticity radius at level-$k$ are, essentially

\[
r_k =  \|D^{(k)}u\|_\infty^{-\frac{1}{k+\frac{3}{2}}}  \ \ \ \ \ \  vs. \ \ \ \ \ \ \  \rho_k = \|D^{(k)}u\|_\infty^{-\frac{1}{2\beta-1}\frac{3}{2}\frac{1}{k+\frac{3}{2}}}
\]
and in order to effectively control the evolution of  $\|D^{(k)}u\|_\infty$ and prevent the blow-up via the harmonic measure maximum
principle one needs $r_k \le \rho_k$ (Theorem \ref{th:SparsityRegDk_Hyp}).
Not surprisingly, this takes place at Lions' exponent $\beta = \frac{5}{4}$,
independently of $k$.

\medskip

However, it will transpire that certain monotonicity properties of the `chain of derivatives' are capable of upgrading the scale of the level-$k$ 
analyticity radius $\rho_k$ to
\[
  \|D^{(k)}u\|_\infty^{-\frac{1}{2\beta-1}\frac{1}{k+1}}.
\]
In these scenarios, it is transparent that as soon as $\beta > 1$ the regularity threshold will be reached for $k$ large enough 
(the closer $\beta$ is to 1, the larger $k$ needs to be). This is the topic of the following section.

\section{The ascending chain condition and the improved local-in-time existence}
%

Henceforth, the symbol $\lesssim$ will denote a bound up to an absolute constant and $\| \cdot \|$ the $L^\infty$-norm. The following result 
(\citet{Grujic2020}) is
a general statement on how an assumption on a large enough portion of the chain being `ascending', i.e., the higher-order derivatives
dominating the lower-order derivatives, results in the prolonged time of local existence and -- in turn -- the improved estimate on the
analyticity radius. We provide a sketch of the proof and several key estimates for reference.

\medskip

\begin{theorem}\label{le:AscendDerHyp}
Let $\beta > 1$, $u_0 \in L^2$, $D^{i}u_0 \in L^\infty$ for $0 \le i \le k$, and suppose that
\begin{align}\label{eq:AscDerOrdMod}
\|D^mu_0\|^{\frac{1}{m+1}} \lesssim \cM_{m,k} \,  \|D^{k}u_0\|^{\frac{1}{k+1}}\  \qquad \forall\ \ell\le m\le k
\end{align}
where the constants $\{\cM_{j,k}\}$ and the indices $\ell$ and $k$ satisfy
\begin{align}\label{eq:BkCondMHyp}
\sum_{\ell \le i\le j-\ell} \binom j i  \cM_{i,k}^{i+1}\ \cM_{j-i,k}^{j-i+1} \lesssim \phi(j,k)\  \qquad \forall\ 2\ell\le j\le k
\end{align}
and
\begin{align}\label{eq:BkCondSHyp}
\|u_0\|_2  \sum_{0\le i\le \ell} \binom j i \cM_{\ell,k}^{\frac{(\ell+1)(i+3/2)}{\ell+3/2}} \cM_{j-i,k}^{j-i+1}  \left((k!)^\frac{1}{k+1}\|u_0\|\right)^{\frac{(3/2-1)(\ell-i)}{\ell+3/2}}\lesssim \psi(j,k)\  \qquad \forall\ 2\ell\le j\le k\ 
\end{align}
for some functions $\phi$ and $\psi$.
If 
\begin{equation}\label{T_cond}
T\lesssim \left(\phi(j,k)+\psi(j,k)\right)^{-\frac{2\beta}{2\beta-1}} \|D^{k}u_0\|^{-\frac{2\beta}{(2\beta-1)(k+1)}}
\end{equation}
 then for any $\ell\le j \le k$ the complexified solution has the following upper bound
\begin{align}\label{eq:AscDerUpBddHyp}
&\underset{t\in(0,T)}{\sup}\ \underset{y\in\cD_t}{\sup} \|D^ju(\cdot,y,t)\| + \underset{t\in(0,T)}{\sup}\ \underset{y\in\cD_t}{\sup} \|D^jv(\cdot,y,t)\| \lesssim \|D^ju_0\| +  \|D^{k}u_0\|^{\frac{j+1}{k+1}}
\end{align}
where $\cD_t$ is given by \eqref{eq:AnalDomHyp}. 
\end{theorem}


\begin{proof}

Construct the approximating sequence as follows,
\begin{align*}
&u^{(0)}=0\ ,\quad \pi^{(0)}=0\ ,
\\
&\partial_t u^{(n)}+( -\Delta)^\beta u^{(n)} = -\left(u^{(n-1)}\cdot \nabla\right)u^{(n-1)} - \nabla\pi^{(n-1)}\ ,
\\
&u^{(n)}(x,0)=u_0(x)\ ,\quad \nabla\cdot u^{(n)}=0\ ,
\\
&\Delta\pi^{(n)}=-\partial_j\partial_k\left(u_j^{(n)} u_k^{(n)}\right)\ .
\end{align*}
By an induction argument (c.f. \citet{Guberovic2010}), $u^{(n)}(t) \in C([0,T], L^\infty(\bR^d))$,  $\pi^{(n)}(t) \in C([0,T], BMO)$, and $u^{(n)}(t)$ and $\pi^{(n)}(t)$ are real analytic for every $t\in (0,T]$ (for any $T>0$). Let $u^{(n)}(x,y,t)+iv^{(n)}(x,y,t)$ and $\pi^{(n)}(x,y,t)+i\rho^{(n)}(x,y,t)$ be the analytic extensions of $u^{(n)}$ and $\pi^{(n)}$ respectively. Then the real and the imaginary parts satisfy
\begin{align}
\partial_t u^{(n)}+(-\Delta)^\beta u^{(n)} &= -\left(u^{(n-1)}\cdot \nabla\right)u^{(n-1)} + \left(v^{(n-1)}\cdot \nabla\right)v^{(n-1)} - \nabla\pi^{(n-1)}\ ,
\\
\partial_t v^{(n)}+(-\Delta)^\beta v^{(n)} &= -\left(u^{(n-1)}\cdot \nabla\right)v^{(n-1)} - \left(v^{(n-1)}\cdot \nabla\right)u^{(n-1)} - \nabla\rho^{(n-1)}\ 
\end{align}
where
\begin{align*}
\Delta \pi^{(n)} &= -\partial_j\partial_k\left(u^{(n)}_ju^{(n)}_k-v^{(n)}_jv^{(n)}_k\right),\qquad \Delta\rho^{(n)} = -2\partial_j\partial_k\left(u^{(n)}_jv^{(n)}_k\right)\ .
\end{align*}
In order to track the expansion of the domain of analyticity in the imaginary directions, define (c.f. \citet{Grujic1998})
\begin{align*}
U_\alpha^{(n)}(x,t)&= u^{(n)}(x,\alpha t,t), && \Pi_\alpha^{(n)}(x,t)= \pi^{(n)}(x,\alpha t,t), 
\\
V_\alpha^{(n)}(x,t)&= v^{(n)}(x,\alpha t,t), && R_\alpha^{(n)}(x,t)= \rho^{(n)}(x,\alpha t,t);
\end{align*}
then the approximation scheme becomes (for simplicity we drop the subscript $\alpha$)
\begin{align*}
\partial_t U^{(n)}+(- \Delta)^\beta U^{(n)} &= -\alpha\cdot \nabla V^{(n)}-\left(U^{(n-1)}\cdot \nabla\right)U^{(n-1)} + \left(V^{(n-1)}\cdot \nabla\right)V^{(n-1)} - \nabla\Pi^{(n-1)}\ ,
\\
\partial_t V^{(n)}+(- \Delta)^\beta V^{(n)} &= -\alpha\cdot \nabla U^{(n)}-\left(U^{(n-1)}\cdot \nabla\right)V^{(n-1)} - \left(V^{(n-1)}\cdot \nabla\right)U^{(n-1)} - \nabla R^{(n-1)}\ ,
\\
\Delta \Pi^{(n)} &= -\partial_j\partial_k\left(U^{(n)}_jU^{(n)}_k-V^{(n)}_jV^{(n)}_k\right),\qquad \Delta R^{(n)} = -2\partial_j\partial_k\left(U^{(n)}_jV^{(n)}_k\right)\ 
\end{align*}
supplemented with the initial conditions
\begin{align*}
U^{(n)}(x,0)=u_0(x),\qquad V^{(n)}(x,0)=0 \qquad\textrm{for all }x\in\bR^3\ .
\end{align*}
This -- via Duhamel -- leads to the following set of iterations,
\begin{align}
D^jU^{(n)}(x,t)& =G_t^{(\beta)}*D^ju_0 - \int_0^t G_{t-s}^{(\beta)}*\nabla D^j\left(U^{(n-1)}\otimes U^{(n-1)}\right) ds + \int_0^t G_{t-s}^{(\beta)}*\nabla D^j\left(V^{(n-1)}\otimes V^{(n-1)}\right) ds \notag
\\
&\quad -\int_0^t G_{t-s}^{(\beta)}*\nabla D^j\Pi^{(n-1)}ds - \int_0^t G_{t-s}^{(\beta)}*\alpha\cdot\nabla D^jV^{(n)}ds\ , \label{eq:IterationHypDU}
\\
D^jV^{(n)}(x,t)& = - \int_0^t G_{t-s}^{(\beta)}* D^j\left(U^{(n-1)}\cdot\nabla\right)V^{(n-1)} ds - \int_0^t G_{t-s}^{(\beta)}* D^j\left(V^{(n-1)}\cdot\nabla\right)U^{(n-1)} ds \notag
\\
&\quad -\int_0^t G_{t-s}^{(\beta)}*\nabla D^j R^{(n-1)}ds - \int_0^t G_{t-s}^{(\beta)}*\alpha\cdot\nabla D^jU^{(n)}ds \label{eq:IterationHypDV}
\end{align}
where $G_t^{(\beta)}$ denotes the fractional heat kernel of order $\beta$.

\medskip

Let
\begin{align*}
K_n &:=\underset{t<T}{\sup}\ \|U^{(n)}\|_{L^2}+\underset{t<T}{\sup}\ \|V^{(n)}\|_{L^2}
\end{align*}
and
\vspace{-0.05in}
\begin{align*}
L_n^{(m)} &:= \underset{t<T}{\sup}\ \|D^mU^{(n)}\|+\underset{t<T}{\sup}\ \|D^mV^{(n)}\|\ , \ \ell\le m \le k\ .
\end{align*}

\medskip

At this point, if one proceeds with the standard estimates and  -- in particular -- use the classical Gagliardo-Nirenberg
interpolation inequalities to estimate the lower-order terms, one arrives at Theorem \ref{th:MainThmVelHyp}. In what follows, we take an
alternative route and replace -- in a large enough portion of the chain -- the classical interpolation inequalities with the ascending 
chain inequalities. 

\medskip

First, we show that -- under a suitable condition -- the assumption \eqref{eq:AscDerOrdMod} on the real parts $U^{(n)}$ will carry over to the imaginary parts $V^{(n)}$. For the basis of induction, notice that
\begin{align}\
\norm{D^mV^{(0)}(x,t)}& = \norm{ \int_0^TG_{t-s}^{(\beta)}*\alpha\cdot\nabla D^mU^{(0)}ds} \notag \\
& = \norm{ \int_0^T\nabla G_{t-s}^{(\beta)}*\alpha\cdot\nabla D^mU^{(0)}ds} \notag\\ 
&\lesssim |\alpha|T^{1-\frac{1}{2\beta}} \norm{D^m U^{(0)}}\notag \\
 &\lesssim |\alpha|T^{1-\frac{1}{2\beta}} M_{m,k}^{m+1} \norm{D^{k} u_0}^{\frac{m+1}{k+1}}. 
\end{align} 
Hence -- assuming $|\alpha|T^{1-\frac{1}{2\beta}} \leq 1/2$ -- yields
\begin{align}\label{eq:AscDerOrdMod_V}
\norm{D^mV^{(0)}(x,t)}& \lesssim M_{m,k}^{m+1} \norm{D^{k} u_0}^{\frac{m+1}{k+1}}
\end{align} 
for $2l\leq m\leq k$. For the inductive step, let us assume that \eqref{eq:AscDerOrdMod} holds for $\{U^{(0)}, U^{(1)}, \dots, U^{(n-2)}\} $, $\{V^{(0)}, V^{(1)}, \dots, V^{(n-2)}\}$ and show that it holds for, e.g., $V^{(n-1)}$. Consider, e.g., the first term on the right-hand side of  \eqref{eq:IterationHypDV}. A straightforward
calculation gives
\begin{align*}
&\left\|\int_0^tG_{t-s}^{(\beta)}*D^m\left(U^{(n-2)}\cdot\nabla\right)V^{(n-2)} ds\right\| \notag \\
&\lesssim T^{1-\frac{1}{2\beta}} 2^{m-1} \|U^{(n-2)}\|_2^{(m-i)/(k+\frac{3}{2})}\|V^{(n-2)}\|_2^{i/(m+\frac{3}{2})}
\|D^mU^{(n-2)}\|^{(i+\frac{3}{2})/(m+\frac{3}{2})} \|D^mV^{(n-2)}\|^{(m-i+\frac{3}{2})/(m+\frac{3}{2})} \notag \\
& \lesssim T^{1-\frac{1}{2\beta}} 2^{k-1}\norm{u_0}_2^{\frac{j}{j+\frac{d}{2}}} M_{m,k}^{m+1} \norm{D^{k}u_0}^{\frac{m+1}{k+1}}
\end{align*}
for any $l\leq m\leq k$. Choosing $T$ as in  \eqref{T_cond} then yields
\begin{align}
&\left\|\int_0^tG_{t-s}^{(\beta)}*D^m\left(U^{(n-2)}\cdot\nabla\right)V^{(n-2)} ds\right\|
\lesssim M_{m,k}^{m+1} \norm{D^{k} u_0}^{\frac{m+1}{k+1}}
\end{align}
for any $l\leq m\leq k$. Estimating the other terms in \eqref{eq:IterationHypDV} in a similar fashion (with a suitable modification in the 
case of the complexified pressure term) leads to 
\begin{align} \label{V_n-1_estimate}
\norm{D^mV^{(n-1)}}_\infty& \lesssim |\alpha|T^{1-\frac{1}{2\beta}} M_{m,k}^{m+1} \norm{D^mU^{(n-1)}}^{\frac{m+1}{k+1}} + M_{m,k}^{m+1} \norm{D^{k}u_0}^{\frac{m+1}{k+1}}\notag\\
&\lesssim M_{m,k}^{m+1} \norm{D^{k}u_0}^{\frac{m+1}{k+1}}
\end{align}
for any $l\leq m\leq k$. 
The estimates on $U^{(n-1)}$ are analogous and one arrives at
\begin{align} \label{eq:AscDerOrdMod_V}
L_{n-1}^{(m)} 
&\lesssim M_{m,k}^{m+1} \norm{D^{k}u_0}^{\frac{m+1}{k+1}}
\end{align}
for any $l\leq m\leq k$.

\medskip

Next, we use assumption \eqref{eq:AscDerOrdMod} and estimate \eqref{eq:AscDerOrdMod_V} to improve the local-in-time result for the
complexified solutions.
We demonstrate the argument on $U^{(n)}\otimes U^{(n)}$ (the rest of the nonlinear terms can be treated in a similar way) via an induction argument. 
For $j>2\ell$,
\begin{align*}
&\left\|\int_0^t G_{t-s}^{(\beta)}*D^j(U^{(n)}\cdot \nabla)U^{(n)}ds\right\| \lesssim t^{1-\frac{1}{2\beta}} \sum_{i=0}^j \binom j i \sup_{s<T}\|D^iU^{(n-1)}(s)\| \sup_{s<T}\|D^{j-i}U^{(n-1)}(s)\|
\\
&\quad \lesssim t^{1-\frac{1}{2\beta}} \left(\sum_{0\le i\le \ell} + \sum_{\ell \le i\le j-\ell} + \sum_{j-\ell\le i\le j}\right) \left(\binom j i \sup_{s<T}\|D^iU^{(n-1)}(s)\| \sup_{s<T}\|D^{j-i}U^{(n-1)}(s)\| \right)
\\
&\quad \lesssim t^{1-\frac{1}{2\beta}} \left(\sum_{\ell \le i\le j-\ell} \binom j i \sup_{s<T}\|D^iU^{(n-1)}(s)\| \sup_{s<T}\|D^{j-i}U^{(n-1)}(s)\| \right.
\\
&\quad \left.+ 2\sum_{0\le i\le \ell} \binom j i \left(\sup_{s<T}\|U^{(n-1)}(s)\|_2\right)^{\frac{\ell-i}{\ell+3/2}} \left(\sup_{s<T}\|D^\ell U^{(n-1)}(s)\|\right)^{\frac{i+3/2}{\ell+3/2}} \sup_{s<T}\|D^{j-i}U^{(n-1)}(s)\|\right)
\\
&\quad \lesssim t^{1-\frac{1}{2\beta}} \left(\sum_{\ell \le i\le j-\ell} \binom j i L_{n-1}^{(i)}L_{n-1}^{(j-i)}+ 2\sum_{0\le i\le \ell} \binom j i K_{n-1}^{\frac{\ell-i}{\ell+3/2}}\left(L_{n-1}^{(\ell)}\right)^{\frac{i+3/2}{\ell+3/2}}L_{n-1}^{(j-i)} \right) =: t^{1-\frac{1}{2\beta}} \left(I+2J\right).
\end{align*}

\medskip

For $I$, \eqref{eq:AscDerOrdMod_V}, \eqref{eq:AscDerOrdMod} and \eqref{eq:BkCondMHyp} yield
\begin{align*}
I &\lesssim \sum_{\ell \le i\le j-\ell} \binom j i\  \cM_{i,k}^{i+1} \|D^{k} u_0\|^{\frac{i+1}{k+1}}\  \cM_{j-i,k}^{j-i+1} \|D^{k} u_0\|^{\frac{j-i+1}{k+1}}
\lesssim \phi(j,k)\|D^{k} u_0\|^{\frac{j+2}{k+1}}.
\end{align*}

\medskip

For $J$, \eqref{eq:AscDerOrdMod_V}, \eqref{eq:BkCondSHyp} and  $\|D^{k}u_0\|\lesssim k!\|u_0\|^{k+1}$ (this without loss of generality) yield
\begin{align*}
J &\lesssim \sum_{0\le i\le \ell} \binom j i\  \|u_0\|_2^{\frac{\ell-i}{\ell+3/2}} \left(\cM_{\ell,k}^{\ell+1} \|D^{k} u_0\|^{\frac{\ell+1}{k+1}}\right)^{\frac{i+3/2}{\ell+3/2}}  \cM_{j-i,k}^{j-i+1}  \|D^{k} u_0\|^{\frac{j-i+1}{k+1}}
\\
&\lesssim \|u_0\|_2\|D^{k} u_0\|^{\frac{j+2}{k+1}} \sum_{0\le i\le \ell} \binom j i \cM_{\ell,k}^{\frac{(\ell+1)(i+3/2)}{\ell+d/2}} \cM_{j-i,k}^{j-i+1} \|D^{k} u_0\|^{\frac{(3/2-1)(\ell-i)}{\ell+3/2}\frac{1}{k+1}}
\\
&\lesssim \|u_0\|_2\|D^{k} u_0\|^{\frac{j+2}{k+1}} \sum_{0\le i\le \ell} \binom j i \cM_{\ell,k}^{\frac{(\ell+1)(i+3/2)}{\ell+d/2}} \cM_{j-i,k}^{j-i+1}  \left((k!)^{\frac{1}{k+1}}\|u_0\|\right)^{\frac{(3/2-1)(\ell-i)}{\ell+3/2}}
\\
& \lesssim \psi(j,k)\|D^{k} u_0\|^{\frac{j+2}{k+1}}.
\end{align*}

\medskip

Treating the other terms in \eqref{eq:IterationHypDU} in a similar fashion gives

\begin{align}\label{asccending_condition_continued}
\|D^jU_n(t)\| 
&\lesssim \|D^ju_0\| + t^{1-\frac{1}{2\beta}} \left(I+2J\right) \lesssim \|D^ju_0\| + t^{1-\frac{1}{2\beta}} (\phi(j,k)+\psi(j,k))\|D^{k} u_0\|^{\frac{j+2}{k+1}}.
\end{align}

\medskip

Hence, as long as $T^{1-\frac{1}{2\beta}}\lesssim (\phi(j,k)+\psi(j,k))^{-1} \|D^{k}u_0\|^{-\frac{1}{k+1}}$, 
\begin{align*}
\sup_{s<t}\|D^jU_n(s)\| 
\lesssim \|D^ju_0\| + \|D^{k} u_0\|^{\frac{j+1}{k+1}}
\end{align*}
as desired. Similarly, with the same condition on $T$,
\begin{align*}
\sup_{s<t}\|D^jV_n(s)\| \lesssim \|D^{k} u_0\|^{\frac{j+1}{k+1}}
\end{align*}
completing the estimate.

\medskip

A standard convergence argument completes the proof  (see \citet{Grujic1998} and \citet{Guberovic2010} for more details). 

\end{proof}

\medskip

The generality of the above result was needed in building the theory of regularity of the 3D HD NS system -- as
soon as $\beta > 1$ -- in a `turbulent regime' presented in \citet{Grujic2020}. Here, we will specify the multiplicative 
coefficients $\cM_{i, k}$ in the ascending chain condition \eqref{eq:AscDerOrdMod} 
to a simple form compatible with the analytic structure, i.e., 
\[
 M_{i,k} = c_0 \frac{(i!)^{\frac{1}{i+1}}}{(k!)^{\frac{1}{k+1}}}
\]
for some constant $c_0 \ge 1$.
This will -- in particular -- yield an explicit condition on the size of the portion of the chain needed to satisfy the
conditions \eqref{eq:BkCondMHyp} and \eqref{eq:BkCondSHyp} and -- in turn -- complete the estimates in the 
previous theorem. Some of the calculations to follow can be optimized further, however, the emphasis here
is on simplicity and transparency.

\medskip

\begin{corollary}\label{accending_chain_Hyp}
Let $\beta > 1, \delta_0 >0$, $u_0 \in L^2$, $k$ a positive integer and $D^i u_0 \in L^\infty$ for $0 \le i \le k$.
Suppose that there exists  a constant $c_0 \ge 1$ such that
\vspace{-0.06in}
\begin{align}\label{eq:AscDerOrd-Hyp}
\|D^ju_0\|^{\frac{1}{j+1}} \leq c_0 \frac{(j!)^{\frac{1}{j+1}}}{(k!)^{\frac{1}{k+1}}}\|D^{k}u_0\|^{\frac{1}{k+1}}\, \qquad  \ell\le j\le k
\end{align}
where $l$ and $k$ satisfy
\vspace{-0.1in}
\begin{align}\label{eq:AscDerCond-Hyp}
\ell! \leq \sqrt{\|u_0\|} \leq (k!)^\frac{1}{k+1}.
\end{align}
Fix $l \le j \le k$ and let $T_j = \frac{1}{c^*}  \,  \left((j!)^{k-j}\|D^{k}u_0\|\right)^{-\frac{2\beta}{(2\beta-1)(k+1)}}$
for a suitable $c_* = c_*(\|u_0\|_2, \beta, c_0, \delta_0)$.
Then the complexified solution has the following upper bound,
\begin{align}\label{eq:AscDerUpBddHyp}
&\underset{t\in(0,T_j)}{\sup}\ \underset{y \in \Omega_t}{\sup} \|D^ju(\cdot , y, t)\| + \underset{t\in(0,T_j)}{\sup}\ \underset{y \in \Omega_t}{\sup} \|D^jv(\cdot , y, t)\| 
\leq \|D^j u_0\| + \delta_0 \frac{(j!)^{\frac{1}{j+1}}}{(k!)^{\frac{1}{k+1}}} \|D^{k}u_0\|^{\frac{j+1}{k+1}}
\end{align}
where the region of analyticity $\Omega_t$  is given by 
\[
\Omega_t=: \left\{ z=x+iy \in \mathbb{C}^3 \ \big|\ |y|\le c \, t^{\frac{1}{2\beta}}\right\}.
\]

\end{corollary}
\begin{proof} 
It is enough to check the conditions \eqref{eq:BkCondMHyp} and \eqref{eq:BkCondSHyp}. 

\medskip

For  \eqref{eq:BkCondMHyp}, notice that 
\begin{align}
\sum_{\ell \le i\le j-\ell} \binom j i  \cM_{i,k}^{i+1}\ \cM_{j-i,k}^{j-i+1} &= \sum_{\ell \le i\le j-\ell} \frac{ j!}{i!(j-i)!} \frac{i!}{(k!)^{\frac{i+1}{k+1}}}\frac{(j-i)!}{(k!)^{\frac{j-i+1}{k+1}}} \notag\\
& =  \sum_{\ell \le i\le j-\ell} \frac{j!}{(k!)^{\frac{j+2}{k+1}}} \notag\\
& \leq \frac{j j!}{(k!)^{\frac{j+2}{k+1}}}\notag \\
& \lesssim \frac{(j!)^{1+ \frac{1}{j+1} -\frac{j+1}{k+1}}}{(k!)^{\frac{1}{k+1}}}\notag \\
& \lesssim \frac{(j!)^{\frac{1}{j+1}}}{(k!)^{\frac{1}{k+1}}} (j!)^{\frac{k-j}{k+1}}.
\end{align}

\medskip

For \eqref{eq:BkCondSHyp}, notice that
\begin{align*} 
& \|u_0\|_2 \sum_{0\le i\le \ell} \frac{j!}{i!(j-i)!} \left( \frac{(l!)^\frac{1}{l+1}}{(k!)^\frac{1}{k+1}}\right)^\frac{(i+3/2)(l+1)}{l+3/2} \frac{(j-i)!}{(k!)^\frac{j-i+1}{k+1}} \left((k!)^\frac{1}{k+1}\|u_0\|\right)^{\frac{\ell-i}{2(\ell+3/2)}} \notag \\
&\qquad =  \|u_0\|_2  \sum_{0\le i\le \ell}\frac{ j!(l!)^\frac{i+3/2}{l+3/2}}{i!} \left(\frac{1}{(k!)^{\frac{1}{k+1}}}\right)^{(l+1)\frac{i+3/2}{l+3/2} + j-i+1} (k!)^{\frac{1}{2(k+1)}\left(1-\frac{i+3/2}{l+3/2}\right)}\|u_0\|^{\frac{1}{2}\left(1-\frac{i+3/2}{l+3/2}\right)}\notag\\
& \qquad = \|u_0\|_2 \frac{ j! \sqrt{\|u_0\|}}{(k!)^{\frac{j+2}{k+1}}} \sum_{0\le i\le \ell} \frac{1}{i!}\left(\frac{l!}{\sqrt{\|u_0\|}}\right)^{\frac{i+3/2}{l+3/2}}.
\end{align*}

\medskip

Hence, under the condition \eqref{eq:AscDerCond-Hyp},
\begin{align} 
&\|u_0\|_2 \sum_{0\le i\le \ell} \binom j i \cM_{\ell,k}^{\frac{(\ell+1)(i+3/2)}{\ell+3/2}} \cM_{j-i,k}^{j-i+1} \left((k!)^\frac{1}{k+1}\|u_0\|\right)^{\frac{(3/2-1)(\ell-i)}{\ell+3/2}} \notag \\ 
&\qquad \lesssim \|u_0\|_2 \frac{1}{(k!)^{\frac{1}{k+1}}}( j!)^\frac{k-j}{k+1} \lesssim \|u_0\|_2  \frac{(j!)^{\frac{1}{j+1}}}{(k!)^{\frac{1}{k+1}}} (j!)^{\frac{k-j}{k+1}}.
\end{align}

\medskip

Consequently, in the notation of the previous theorem, $\phi(j,k) = \frac{(j!)^{\frac{1}{j+1}}}{(k!)^{\frac{1}{k+1}}} (j!)^{\frac{k-j}{k+1}}$, $\psi(j,k) = \|u_0\|_2\,\phi(j,k)$,
and
\[
\|D^ju(t)\| \le \|D^ju_0\| + c_1 \, t^{1-\frac{1}{2\beta}} (\phi(j,k)+\psi(j,k))\|D^{k} u_0\|^{\frac{j+2}{k+1}}
\]
where $c_1$ is a constant depending on $c_0$.
The choice of $T_j$ as in the statement of the corollary yields the desired conclusion.
\end{proof} 


\section{Proof of Theorem 1.1}

\begin{proof}
In what follows, the $L^\infty$-norms will be the $L^\infty$-norms on a ball centered at $x^*$ and contained in $\cal{N}$ -- the neighborhood 
in which the focussing assumption (A2) holds for any $t \in (T^*-\epsilon, T^*)$. Let $k$ be a positive integer and $0 \le j \le k$. By Taylor's theorem

\[
 c^i_{\alpha,k}(t) = \frac{1}{\alpha !} \frac{\partial^k}{\partial^\alpha x} u^i(x^*,t).
\]

Utilizing (A1)-(A2) yields

\[
 \frac{\|D^{(j)} u(t) \|^\frac{1}{j+1}}{\|D^{(k)} u(t) \|^\frac{1}{k+1}}
 \le d  \ M^{\frac{j}{j+1}+\frac{k}{k+1}} \ \rho(t)^{\frac{k}{k+1}-\frac{j}{j+1}} \ \frac{(j!)^\frac{1}{j+1}}{(k!)^\frac{1}{k+1}}
\]
where $d$ is an absolute constant.

\medskip

Since $\rho$ goes to 0 and $\frac{k}{k+1}-\frac{j}{j+1} \ge 0$, for $\epsilon$ small enough, the right hand side
will be bounded by

\[
 d   M^2 \ \frac{(j!)^\frac{1}{j+1}}{(k!)^\frac{1}{k+1}},
\]
i.e., the ascending chain condition \eqref{eq:AscDerOrd-Hyp} is satisfied throughout the chain (with $c_0 = d M^2$).

\medskip

Fix $k$, let $t$ be an escape time for $D^k u$, evolve the system from $t$, and let $s = t + T_k$.

\medskip

Setting $j=k$ in the estimates obtained in the corollary yields 

\medskip

\[
T_k = \frac{1}{c^*}  \, \|D^{k}u(t)\|^{-\frac{2\beta}{(2\beta-1)(k+1)}}
\]
for a suitable $c_* = c_*(\|u_0\|_2, \beta, d, M, \delta_0)$ and

\medskip

\[
 \underset{\tau\in(t, s)}{\sup}\ \underset{y \in \Omega_\tau}{\sup} \|D^ku(\cdot , y, \tau)\| + \underset{\tau\in(t, s)}{\sup}\ \underset{y \in \Omega_\tau}{\sup} \|D^kv(\cdot , y, \tau)\| \leq (1+\delta_0) \|D^k u(t)\|.
\] 

\medskip

Recall that in order to prevent the blow-up via the harmonic measure maximum principle (see \citet{Grujic2013, Bradshaw2019} 
in the case of the velocity and the vorticity fields, respectively)
the scale of the radius of spatial analyticity at $s$, $\rho_k$ needs to dominate the \emph{a priori} scale of sparseness at $s$, $r_k$.
In other words, at the level $k$, a natural small scale associated with the regions of the intense fluid activity -- the scale of sparseness of the
suitably cut super-level sets of the/a maximal component of $D^{(k)}u$ -- needs to fall into the level $k$ diffusion range represented by the lower
bound on the radius of spatial analyticity. This is precisely the regularity criterion described in Theorem \ref{th:SparsityRegDk_Hyp}, 
except that the general estimate on
the analyticity radius is now replaced with the improved estimate obtained in the corollary. 
In particular, the multiplicative 
constant in the estimate on the complexified solution, $1+\delta_0$ ($M$ in Theorem \ref{th:SparsityRegDk_Hyp}) needs to satisfy a suitable 
algebraic inequality 
originating in the calculation of the harmonic measure, this
determines $\delta_0$.

\medskip

Hence, 
$r_k$ \emph{vs.}  $\rho_k$ is now

\[
r_k = c_1(\|u_0\|_2) \, \|D^{(k)}u\|^{-\frac{1}{k+\frac{3}{2}}}  \ \ \ \ \  vs. \ \ \ \ \ 
\rho_k = \frac{1}{c_2(\|u_0\|_2, \beta, d, M, \delta_0)} \,  \|D^{(k)}u\|^{-\frac{1}{2\beta-1}\frac{1}{k+1}}.
\]

\medskip

Notice that the gap in the exponents, 

\[
 \frac{1}{k+3/2} - \frac{1}{(2\beta -1) (k+1)}  = \frac{(2\beta -2)k +(2\beta -5/2)}{(2\beta-1)(k+1)(k+3/2)}
\]
is positive for any $k > \frac{2\beta-5/2}{2-2\beta}$ which is positive for any $\beta$ in the super-critical regime $\beta \in (1, 5/4)$.
As expected, the closer $\beta$ to 1, the large $k$ needs to be. Choosing $k \ge k^*(\|u_0\|_2, \beta, d, M, \delta_0)$ for a 
suitable $k^*$ , (A1) will assure that $\rho_k$ dominates $r_k$.

\medskip

The last thing to check is the range of indices needed in the ascending condition given by the inequality
\eqref{eq:AscDerCond-Hyp}, essentially
\[
 \|u(t)\| \le k^2
\]
(take $l=1$).
\medskip

So far, we have not used the assumption on monotonicity of the coefficients in the Taylor expansions of the blow-up 
profile (checking the ascending chain condition, i.e., monotonicity of the derivatives required only $\rho(t)$ shrinking to 0 as 
we approach the singular time). At this point, monotonicity of the coefficients will give us a quick way to close the 
argument. By monotonicity, any time $t$ in $(T^*-\epsilon, T^*)$ is an escape time for any level $k$. 
Fix $t$ first, then the requirement  $\|u(t)\| \le k^2$ becomes just another lower bound on $k$ (in addition to $k^*$).
\end{proof}

\medskip

\begin{remark}
If one does not assume monotonicity of the coefficients, coordinating the condition assuring that the range of indices is
large enough with the condition needed to assure $\rho_k \ge r_k$ becomes more subtle. In the general case
(\citet{Grujic2020}) dynamics of the chain is deconstructed in monotone pieces (ascending or descending) and the 
`undecided' pieces. Since the utility of the ascending portions of the chain is replacing the classical
Gagliardo-Nirenberg interpolation inequalities, one can view it as `dynamic interpolation'. In the descending portions,
one starts with the general lower bound on the analyticity radius (Theorem \ref{th:MainThmVelHyp}), and then uses
the descending condition in conjunction with the \emph{a priori} sparseness to extend the solution analytically via 
Taylor series. The monotone (`turbulent') regime is then demonstrated to be singularity free for any $\beta > 1$.
In the present work the emphasis is on 
simplicity and clarity in demonstrating how monotonicity of the chain (the ascending case) can rule out a possible formation of
singularities as soon as the hyper-diffusion exponent  is greater than 1, under an additional assumption on monotonicity
of the blow-up profile (a `runaway train' scenario).
\end{remark}

\section{Acknowledgments}

A.F. and Z.G. acknowledge the support of the National Science Foundation via the grants DMS 2206493 and DMS 2307657, respectively.

\bibliographystyle{abbrvnat}
\bibliographystyle{plainnat}

\begin{thebibliography}{19}
\providecommand{\natexlab}[1]{#1}
\providecommand{\url}[1]{\texttt{#1}}
\expandafter\ifx\csname urlstyle\endcsname\relax
  \providecommand{\doi}[1]{doi: #1}\else
  \providecommand{\doi}{doi: \begingroup \urlstyle{rm}\Url}\fi







\bibitem[Bradshaw et~al.(2019)Bradshaw, Farhat, and Gruji\'{c}]{Bradshaw2019}
Z.~Bradshaw, A.~Farhat, and Z.~Gruji\'{c}.
\newblock An {A}lgebraic {R}eduction of the `{S}caling {G}ap' in the
  {N}avier--{S}tokes {R}egularity {P}roblem.
\newblock \emph{Arch. Ration. Mech. Anal.}, 231\penalty0 (3):\penalty0
  1983--2005, 2019.
\newblock ISSN 0003-9527.
\newblock \doi{10.1007/s00205-018-1314-5}.


\bibitem[Farhat et~al.(2017)Farhat, Gruji\'{c}, and Leitmeyer]{Farhat2017}
A.~Farhat, Z.~Gruji\'{c}, and K.~Leitmeyer.
\newblock The space {$B^{-1}_{\infty,\infty}$}, volumetric sparseness, and 3{D}
  {NSE}.
\newblock \emph{J. Math. Fluid Mech.}, 19\penalty0 (3):\penalty0 515--523,
  2017.
\newblock ISSN 1422-6928.
\newblock \doi{10.1007/s00021-016-0288-z}.


\bibitem[Gruji\'{c}(2013)]{Grujic2013}
Z.~Gruji\'{c}.,
\newblock A geometric measure-type regularity criterion for solutions to the
  3{D} {N}avier-{S}tokes equations.
\newblock \emph{Nonlinearity}, 26\penalty0 (1):\penalty0 289--296, 2013.
\newblock ISSN 0951-7715.
\newblock \doi{10.1088/0951-7715/26/1/289}.


\bibitem[Gruji\'{c} and Kukavica(1998)]{Grujic1998}
Z.~Gruji\'{c} and I.~Kukavica.
\newblock Space analyticity for the {N}avier-{S}tokes and related equations
  with initial data in {$L^p$}.
\newblock \emph{J. Funct. Anal.}, 152\penalty0 (2):\penalty0 447--466, 1998.
\newblock ISSN 0022-1236.
\newblock \doi{10.1006/jfan.1997.3167}.




\bibitem[Gruji\'{c} and Xu(2019)]{Grujic2019}
Z.~Gruji\'{c} and L.~Xu.
\newblock Asymptotic criticality of the Navier-Stokes regularity problem.
\newblock \emph{Preprint \url{https://arxiv.org/abs/1911.00974}}, 2019.



\bibitem[Gruji\'{c} and Xu(2020)]{Grujic2020}
Z.~Gruji\'{c} and L.~Xu.
\newblock Time-global regularity of the Navier-Stokes system with hyper-dissipation--turbulent scenario.
\newblock \emph{Preprint \url{https://arxiv.org/abs/2012.05692}}, 2020.



\bibitem[Guberovi\'{c}(2010)]{Guberovic2010}
R.~Guberovi\'{c}.
\newblock Smoothness of {K}och-{T}ataru solutions to the {N}avier-{S}tokes
  equations revisited.
\newblock \emph{Discrete Contin. Dyn. Syst.}, 27\penalty0 (1):\penalty0
  231--236, 2010.
\newblock ISSN 1078-0947.
\newblock \doi{10.3934/dcds.2010.27.231}.


\bibitem[Iyer et~al.(2014)Iyer, Kiselev, and Xu]{Iyer2014}
G.~Iyer, A.~Kiselev, and X.~Xu.
\newblock Lower bounds on the mix norm of passive scalars advected by
  incompressible enstrophy-constrained flows.
\newblock \emph{Nonlinearity}, 27\penalty0 (5):\penalty0 973--985, 2014.
\newblock ISSN 0951-7715.
\newblock \doi{10.1088/0951-7715/27/5/973}.


\bibitem[Lions(1959)]{Lions1959}
J.-L. Lions.
\newblock Quelques r\'{e}sultats d'existence dans des \'{e}quations aux
  d\'{e}riv\'{e}es partielles non lin\'{e}aires.
\newblock \emph{Bull. Soc. Math. France}, 87:\penalty0 245--273, 1959.
\newblock ISSN 0037-9484.


\bibitem[Lions(1969)]{Lions1969}
J.-L. Lions.
\newblock \emph{Quelques m\'{e}thodes de r\'{e}solution des probl\`emes aux
  limites non lin\'{e}aires}.
\newblock Dunod; Gauthier-Villars, Paris, 1969.



\bibitem[Ransford(1995)]{Ransford1995}
T.~Ransford.
\newblock \emph{Potential theory in the complex plane}, volume~28 of
  \emph{London Mathematical Society Student Texts}.
\newblock Cambridge University Press, Cambridge, 1995.
\newblock ISBN 0-521-46120-0; 0-521-46654-7.
\newblock \doi{10.1017/CBO9780511623776}.


\bibitem[Solynin(1997)]{Solynin1997}
A.~Y. Solynin.
\newblock Ordering of sets, hyperbolic metric, and harmonic measure.
\newblock \emph{Zap. Nauchn. Sem. S.-Peterburg. Otdel. Mat. Inst. Steklov.
  (POMI)}, 237\penalty0 (Anal. Teor. Chisel i Teor. Funkts. 14):\penalty0
  129--147, 230, 1997.
\newblock ISSN 0373-2703.
\newblock \doi{10.1007/BF02172470}.



\end{thebibliography}

\def\cprime{$'$}

\end{document}